\documentclass{IEEEtran}
\usepackage[T1]{fontenc}
\usepackage[utf8]{inputenc}
\usepackage{array}
\usepackage{mathtools}
\usepackage{multirow}
\usepackage{amsmath}
\usepackage{amsthm}
\usepackage{amssymb}
\usepackage{graphicx}

\makeatletter

\providecommand{\tabularnewline}{\\}

\theoremstyle{plain}
\newtheorem{thm}{\protect\theoremname}
\theoremstyle{plain}
\newtheorem{lem}[thm]{\protect\lemmaname}
\theoremstyle{plain}
\newtheorem{prop}[thm]{\protect\propositionname}

\usepackage{fontawesome}
\usepackage{color}
\usepackage{cite}
\usepackage{amsthm}
\theoremstyle{plain}
\newtheorem{asm}{Assumption}
\newtheorem*{asm*}{Assumption}
\usepackage{xcolor}
\usepackage{multirow}
\usepackage[normalem]{ulem}

\usepackage{hyperref}

\ifdefined\showcaptionsetup
 \PassOptionsToPackage{caption=false}{subfig}
\fi
\usepackage{subfig}
\makeatother

\providecommand{\lemmaname}{Lemma}
\providecommand{\propositionname}{Proposition}
\providecommand{\theoremname}{Theorem}

\begin{document}
\title{Spectral Initialization and Certification for Power System Angle Estimation}
\author{Iven Guzel, \emph{Student Member, IEEE}, Andrew~D.~McRae, and Richard
Y. Zhang,\emph{ Member, IEEE}\thanks{I. Guzel and R. Y. Zhang are with the Dept. of Electrical and Computer
Engineering, University of Illinois Urbana-Champaign, Urbana, IL 61801
USA (e-mail: iguzel2@illinois.edu; ryz@illinois.edu).}\thanks{A. D. McRae is with CERMICS, CNRS, ENPC, Institut Polytechnique de
Paris, Marne-la-Vallée, France. (e-mail: andrew.mcrae@enpc.fr).}}

\maketitle

\global\long\def\H{\mathbb{H}}%
\global\long\def\S{\mathbb{S}}%
\global\long\def\R{\mathbb{R}}%
\global\long\def\C{\mathbb{C}}%

\global\long\def\inner#1#2{\left\langle #1,#2\right\rangle }%
\global\long\def\bigp#1{\left(#1\right)}%
\global\long\def\bigsp#1{\left[#1\right]}%
\global\long\def\bigcp#1{\left\{  #1\right\}  }%
\global\long\def\Cone{\C_{1}^{n}}%
\global\long\def\norm#1{\lVert#1\rVert}%
\global\long\def\opnorm#1{\lVert#1\rVert_{\mathrm{op}}}%

\global\long\def\rmse{\operatorname{RMSE}}%

\global\long\def\Im{\operatorname{Im}}%
\global\long\def\Re{\operatorname{Re}}%
\global\long\def\tr{\operatorname{tr}}%
\global\long\def\proj{\operatorname{proj}}%
\global\long\def\rank{\operatorname{rank}}%
\global\long\def\nullity{\operatorname{nullity}}%
\global\long\def\diag{\operatorname{diag}}%
\global\long\def\Diag{\operatorname{Diag}}%
\global\long\def\vec{\operatorname{vec}}%
\global\long\def\ddiag{\operatorname{ddiag}}%
\global\long\def\conj{\operatorname{conj}}%
\global\long\def\dist{\operatorname{dist}}%
\global\long\def\ker{\operatorname{Ker}}%
\global\long\def\op{\operatorname{op}}%
\global\long\def\grad{\operatorname{grad}}%
\global\long\def\hess{\operatorname{Hess}}%

\global\long\def\obs{\mathrm{obs}}%
\global\long\def\opt{\mathrm{opt}}%
\global\long\def\ub{\mathrm{ub}}%
\global\long\def\lb{\mathrm{lb}}%
\global\long\def\gnd{\mathrm{gnd}}%
\global\long\def\wls{\mathrm{wls}}%
\global\long\def\bus{\mathrm{bus}}%
\global\long\def\From{\mathrm{f}}%
\global\long\def\To{\mathrm{t}}%
\global\long\def\opt{\mathrm{opt}}%
\global\long\def\si{\mathrm{si}}%
\global\long\def\sr{\mathrm{SR}}%

\global\long\def\h#1{\hat{#1}}%
\global\long\def\dt#1{\dot{#1}}%
\global\long\def\t#1{\tilde{#1}}%
\global\long\def\o#1{\overline{#1}}%

\global\long\def\e{\mathbf{e}}%
\global\long\def\i{\mathrm{j}}%
\global\long\def\Aop{\mathcal{A}}%
\global\long\def\one{\mathbf{1}}%
\global\long\def\half{{\textstyle \frac{1}{2}}}%

\begin{abstract}
Power System State Estimation (PSSE) is commonly formulated as a
nonconvex weighted least-squares (WLS) problem, making global optimality
difficult both to attain and to certify. Recent work has shown that,
when voltage magnitudes are known to sufficient accuracy, the remaining
angle estimation subproblem can be reduced to phase synchronization
and solved effectively using spectral initialization and spectral
certification. This paper explains why these spectral methods succeed.
We prove that their behavior is governed by the measurement error,
normalized against the usual observability margin from the classical
literature. Below a fixed threshold, spectral initialization recovers
the true voltage angles to first-order accuracy, and the WLS estimator
is unique up to a global phase. Moreover, a zero-duality-gap spectral
certificate verifies recovery of this unique WLS estimate. In the
noiseless observable regime, spectral initialization exactly recovers
the true angles and certification is exact without local refinement.
Numerical experiments on standard benchmark systems support the theory
and show that the spectral methods remain effective beyond the conservative
regime covered by the guarantees. 
\end{abstract}


\section{Introduction}

The reliable operation of modern electric power grids depends on Power
System State Estimation (PSSE), which provides the real-time network
state used for monitoring, control, and situational awareness \cite{monticelli2000electric,abur2004power,cheng2023survey}.
Its goal is to recover the voltage magnitudes and voltage phase angles
from noisy measurements of voltage magnitudes and complex power. The
standard approach \cite{schweppe1970power} formulates PSSE as a weighted
least-squares (WLS) problem and solves it by iteratively refining
an initial guess using the Gauss--Newton method.

The core challenge is that the WLS objective function is nonconvex,
and this makes global optimality both difficult to attain and difficult
to verify. With a sufficiently accurate initialization, Gauss--Newton
converges rapidly and reliably, but outside that regime, it can converge
slowly, fail to converge, or settle at a spurious local optimum. When
this happens, it is generally difficult to determine whether the failure
is caused by insufficient information in the measurements, for example
due to bad data \cite{merrill1971bad,koglin1990bad}, or by the numerical
behavior of the algorithm itself. Recent work~\cite{zhang2019spurious,guzel2024power}
has further showed that Gauss--Newton can converge to a plausible
but incorrect estimate of the true state. Since local optimality does
not imply global optimality, even a high-quality estimate usually
comes with no rigorous way to check that it is globally optimal. 

Recent work~\cite{guzel2024power} made progress on the nonconvexity
of PSSE by isolating the angle estimation subproblem. When the voltage
magnitudes are known to sufficient accuracy, estimating the remaining
voltage phase angles from noisy complex power measurements reduces
to a phase synchronization problem. This reduction makes spectral
methods applicable \cite{singer2011angular,bandeira2017tightness,boumal2016nonconvex}.
By computing eigenvectors of matrices constructed from the measurements,
one can obtain both \emph{spectral initialization}, a high-quality
initial estimate that is often already close to the true phase angles,
and \emph{spectral certification}, a post hoc certificate that the
recovered point is a globally optimal WLS estimate for the given data.
Experiments in~\cite{guzel2024power} showed that this approach can
solve the angle subproblem to certified global optimality, often with
little or no local refinement, at a cost comparable to a few Gauss--Newton
iterations.

Despite this strong empirical performance, we still do not know \emph{why}
spectral methods work so well for PSSE. The angle subproblem reduces
to a phase synchronization problem, but this reduction is not, by
itself, an explanation. Generic phase synchronization is a nonconvex
quadratic optimization problem whose global optimum is hard to attain
and hard to certify. Moreover, when the measurements are too few or
too noisy, accurate recovery of the angles is fundamentally impossible
for any method. Therefore, explaining the empirical success observed
in~\cite{guzel2024power} requires showing the phase synchronization
instances arising in realistic PSSE settings contain enough information
for near-recovery and admit exact certificates of global optimality.

This paper fills that theoretical gap, by providing a power-system-specific
explanation that relates spectral recovery and certification to observability,
measurement noise, and the structure of the network. Specifically,
we show that the behavior of spectral methods for the PSSE angle subproblem
is governed by a single normalized noise level 
\[
\eta\coloneqq\frac{K_{s}\|\epsilon_{s}\|_{\infty,\mathrm{obs}}+K_{u}\|\epsilon_{u}\|_{\infty}}{\rho},
\]
where $\|\epsilon_{s}\|_{\infty,\mathrm{obs}}$ and $\|\epsilon_{u}\|_{\infty}$
measure the maximum deviations in the observed complex power measurements
and in the voltage magnitude measurements, respectively; $\rho$ is
the observability margin from the classical PSSE literature~\cite{fetzer1975observability,krumpholz1980power};
and $K_{s},K_{u}$ are constants determined by the network and measurement
model. We prove in Section~\ref{sec:init} that when $\eta$ is below
a fixed threshold, spectral initialization and the WLS estimator both
recover the true phase angles to $O(\eta)$ accuracy. In this same
small-$\eta$ setting, we further show in Section~\ref{sec:cert}
that the WLS estimator is unique up to a global phase, and the spectral
certificate is guaranteed to verify whether a recovered point is this
unique WLS estimate, without any conservatism.

The resulting picture is surprisingly simple. In the noiseless observable
regime $\eta=0$, spectral initialization exactly recovers the true
angles, and the spectral certificate is exact without local refinement.
In observable low-noise regimes $\eta\approx0$, spectral initialization
is $O(\eta)$ accurate, and refinement recovers the unique WLS estimator,
with a spectral certificate verifying its recovery. Together, these
results give the first rigorous guarantee for globally solving and
certifying the nonconvex angle-estimation subproblem without a good
initial guess.

Our experiments in Section~\ref{sec:Experiments} test these predictions
directly. Across standard benchmark systems, the spectral initialization
error follows the linear scaling predicted by the theory, and spectral
certification successfully verifies global optimality after local
refinement. As the normalized noise level $\eta$ increases, either
because the observed measurement errors $\|\epsilon_{s}\|_{\infty,\mathrm{obs}}$
and $\|\epsilon_{u}\|_{\infty}$ grow or because the observability
margin $\rho$ deteriorates, the theoretical guarantees eventually
cease to apply. Nevertheless, our experiments observe that the spectral
methods continue to perform well empirically beyond this fairly conservative
regime. At sufficiently high measurement error or low observability,
however, accurate recovery and exact certification must break down,
as the recovery problem eventually becomes impossible to solve. Precisely
characterizing this transition remains an open question.

Our results place PSSE angle estimation within a broader literature
on provable spectral methods for phase synchronization and related
group-synchronization problems. Such models arise in rotation averaging
for computer vision~\cite{hartley2013rotation}, orientation estimation
in single-particle cryo-electron microscopy~\cite{bendory2020single},
rotation averaging and pose-graph SLAM in robotics~\cite{doherty2022performance},
sensor network localization~\cite{cucuringu2012sensor}, and phase
retrieval~\cite{iwen2020phase}. The closest theoretical point of
comparison that we are aware of is the robotics literature on spectral
initialization and certifiable geometric estimation~\cite{doherty2022performance,papalia2024certifiably}.
However, those results analyze different measurement models and do
not capture the power-system-specific roles of complex power injections,
voltage-magnitude perturbations, and the observability margin.

\subsection*{Notation}

Write $\odot$ as the element-wise Hadamard product; $\R_{+}^{n}\subseteq\R^{n}$
as the length-$n$ nonnegative real vectors; $\mathbb{Z}^{n}$ as
length-$n$ integer vectors; $\i\coloneqq\sqrt{-1}$; $\one=[1,1,\dots,1]^{T}$;
$\sigma_{i}(\cdot)$ as the $i$-th largest signular value. Take $|\cdot|$,
$\arg(\cdot),$ and $\exp(\cdot)$ elementwise. $\Diag:\C^{n}\to\C^{n\times n}$
and $\diag:\C^{n\times n}\to\C^{n}$ are the standard diagonal operators;
$\ddiag\coloneqq\Diag\circ\diag$. 

\section{Background}

\subsection{Complex power equations }

In the standard linear AC circuit model, a power network is represented
as a weighted graph with $n$ vertices, called buses, and $l$ edges,
called branches. The vector $v\in\mathbb{C}^{n}$ records the complex
voltage at each bus. Write $i_{\bus}\in\mathbb{C}^{n}$ as the net
complex current injected at each bus, and $i_{\From},i_{\To}\in\mathbb{C}^{l}$
as the currents at the two ends of each branch, with respect to a
fixed orientation of the graph. Since lines, transformers, and shunt
elements are modeled as linear circuit elements, all currents depend
linearly on the voltage vector: 
\begin{equation}
i_{\bus}=Y_{\bus}v,\qquad i_{\From}=Y_{\From}v,\qquad i_{\To}=Y_{\To}v.\label{eq:i_v}
\end{equation}
Here $Y_{\bus}\in\mathbb{C}^{n\times n}$ and $Y_{\From},Y_{\To}\in\mathbb{C}^{l\times n}$
are the standard admittance matrices associated with this circuit
model. We defer their construction from engineering data to standard
references~\cite{zimmerman2010matpower}.

The corresponding complex powers are obtained by multiplying the relevant
voltage phasor by the conjugate of the corresponding current phasor:
\begin{subequations}\label{eq:sdef} 
\begin{gather}
s_{\bus}=\conj(i_{\bus})\odot v=\conj(Y_{\bus}v)\odot v,\\
s_{\From}=\conj(i_{\From})\odot(F_{\From}v)=\conj(Y_{\From}v)\odot(F_{\From}v),\\
s_{\To}=\conj(i_{\To})\odot(F_{\To}v)=\conj(Y_{\To}v)\odot(F_{\To}v).
\end{gather}
\end{subequations}Here $F_{\From},F_{\To}\in\{0,1\}^{l\times n}$
select the voltage at the ``from'' and ``to'' ends of each branch,
respectively, so that $F=F_{\From}-F_{\To}$ is the usual oriented
incidence matrix of the graph.

We now collect the three types of complex power quantities in (\ref{eq:sdef})
into a single vector. Let $m=n+2l$, and define the map $s:\mathbb{C}^{n}\to\mathbb{C}^{m}$
by 
\begin{equation}
s(v)\coloneqq\begin{bmatrix}s_{\bus}(v)\\
s_{\From}(v)\\
s_{\To}(v)
\end{bmatrix}=\begin{bmatrix}\conj(Y_{\bus}v)\odot v\\
\conj(Y_{\From}v)\odot(F_{\From}v)\\
\conj(Y_{\To}v)\odot(F_{\To}v)
\end{bmatrix}.\label{eq:sv_def}
\end{equation}
Equation (\ref{eq:sv_def}) is the basic complex power-flow model
used throughout the paper. In the sequel, the matrices $Y_{\bus},Y_{\From},Y_{\To},F_{\From},F_{\To}$
are treated as known problem data.

\subsection{Weighted least-squares formulation}

The classical power system state estimation problem seeks to recover
an unknown ground-truth voltage phasor $v_{\gnd}\in\mathbb{C}^{n}$
from a partial and noisy collection of complex power measurements
and voltage magnitude measurements. The standard bookkeeping convention
is to work with a complete set of measurements
\begin{equation}
\tilde{s}=s(v_{\gnd})+\epsilon_{s}\in\mathbb{C}^{m},\qquad\tilde{u}=|v_{\gnd}|+\epsilon_{u}\in\mathbb{R}^{n},\label{eq:noisy_meas}
\end{equation}
and later remove unobserved entries by treating their noise variance
as infinite. Define the real-valued measurement map $h:\mathbb{R}^{2n}\to\mathbb{R}^{2m+n}$
and the observed measurement vector $\tilde{h}\in\mathbb{R}^{2m+n}$
as respectively 
\[
h(\theta,u)=\begin{bmatrix}\Re\{s(u\odot\exp(\i\theta))\}\\
\Im\{s(u\odot\exp(\i\theta))\}\\
u
\end{bmatrix},\quad\tilde{h}=\begin{bmatrix}\Re\{\tilde{s}\}\\
\Im\{\tilde{s}\}\\
\tilde{u}
\end{bmatrix}.
\]
Parameterizing $v=u\odot\exp(\mathrm{i}\theta),$ the standard weighted
least-squares estimator solves
\begin{equation}
(\theta_{\wls},u_{\wls})=\arg\min_{(\theta,u)\in\R^{2n}}\|h(\theta,u)-\tilde{h}\|_{w}^{2},\tag{WLS}\label{eq:wls}
\end{equation}
where given nonnegative weights $w\in\mathbb{R}_{+}^{N}$, the weighted
norm on $\R^{N}$ (or $\C^{N}$) is defined as 
\[
\|z\|_{w}\coloneqq\left(\sum_{i=1}^{n}w_{i}|z_{i}|^{2}\right)^{1/2}.
\]
Hence, large weights emphasize reliable measurements, small weights
downweight unreliable measurements, and zero weights implement the
infinite-variance convention for unobserved entries by removing them
from the objective.

\subsection{Ambiguity of Absolute Phase}

The solution of (\ref{eq:wls}) is never unique, because the absolute
phase reference is unobservable: 
\begin{equation}
h(\theta,u)=h(\theta+\theta_{0}\one,u)\qquad\text{for all }\theta_{0}\in\R.\label{eq:abs_phase}
\end{equation}
The standard approach removes this ambiguity by choosing a reference
angle, for example $\theta_{1}=0$ when bus~1 is used as the slack
bus. This is convenient for numerical implementations, but it makes
the formulation depend on an arbitrary reference bus. An alternative,
common in the optimization literature, is to retain the global angle
ambiguity and compare estimates after allowing the best constant angle
shift. Accordingly, we define the root-mean-square angular error
\begin{equation}
\rmse(\theta,\theta_{\gnd})\coloneqq\frac{1}{\sqrt{n}}\min_{\theta_{0}\in\R,\;k\in\mathbb{Z}^{n}}\|\theta-\theta_{\gnd}-\theta_{0}\one-2\pi k\|.\label{eq:rmsedef}
\end{equation}
The factor $1/\sqrt{n}$ makes this an RMS error per bus, measured
in radians. If $\rmse(\theta,\theta_{\gnd})=0$, then the estimated
angles exactly recover the ground truth up to the unknown global phase. 

For theoretical analysis, the same reference-angle ambiguity is more
conveniently expressed in terms of the corresponding unit phasors
$x=\exp(i\theta)$ and $x_{\gnd}=\exp(i\theta_{\gnd})$. Define
\begin{equation}
\dist(x,x_{\gnd})\coloneqq\min_{\theta_{0}\in\R}\|x-x_{\gnd}\exp(\i\theta_{0})\|.\label{eq:distdef}
\end{equation}
This phasor distance and the RMS angular error are equivalent:
\begin{equation}
\frac{1}{\sqrt{n}}\dist(x,x_{\gnd})\le\rmse(\theta,\theta_{\gnd})\le\frac{\pi}{2\sqrt{n}}\dist(x,x_{\gnd}).\label{eq:rmse_dist}
\end{equation}

\subsection{Power System Observability}\label{subsec:observe}

A basic prerequisite for state estimation is that the measurements
contain enough information to determine the unknown state~\cite{fetzer1975observability}.
Classical observability tests~\cite{fetzer1975observability,krumpholz1980power,wu1985network,monticelli1985network}
are based on the weighted measurement Jacobian 
\begin{equation}
H(\theta,u)\coloneqq\diag(w)^{1/2}\begin{bmatrix}\dfrac{\partial h(\theta,u)}{\partial\theta} & \dfrac{\partial h(\theta,u)}{\partial u}\end{bmatrix}.\label{eq:wls_jacobian}
\end{equation}
The network is said to be \emph{observable} at $(\theta_{\gnd},u_{\gnd})$
if this Jacobian has rank $2n-1$~\cite[Appendix~B]{krumpholz1980power}\cite[Thm.~1]{wu1985network},
with the missing rank corresponding to the global phase $(\one,0)$. 

Following Fetzer and Anderson~\cite{fetzer1975observability}, we
quantify the strength of this condition by the \emph{observability
margin}
\begin{equation}
\rho\coloneqq\sigma_{2n-1}\!\left(H(\theta_{\gnd},u_{\gnd})\right).\label{eq:observability_margin}
\end{equation}
The network is observable if and only if $\rho>0$, while the magnitude
of $\rho$ quantifies the local conditioning of the inverse problem
near the ground truth.

\section{Angle Estimation as Phase Sychronization}

The standard approach to solving the weighted least-squares problem
(\ref{eq:wls}) is to apply Gauss--Newton. Guzel and Zhang~\cite{guzel2024power}
observed that, under a particular choice of weights, the same estimation
problem admits an exact reformulation as phase synchronization, hence
making it amenable to spectral methods.

\begin{asm}\label{asm} Partition the weights in (\ref{eq:wls})
as $w=(w_{p},w_{q},w_{u})$, corresponding respectively to active
power, reactive power, and voltage magnitude measurements. Assume
that $w_{u}\to\infty$ and that $w_{p}=w_{q}=d$ for some $d\in\R_{+}^{m}$.\end{asm}

Assumption~\ref{asm} fixes the voltage magnitudes at their measured
values, as in $u=\tilde{u}$, and rewrites the real and reactive power
residuals into a complex power residual
\[
\|\Re\{s(v)\}-\tilde{p}\|_{w_{p}}^{2}+\|\Im\{s(v)\}-\tilde{q}\|_{w_{p}}^{2}=\|s(v)-\tilde{s}\|_{d}^{2}
\]
for $\tilde{s}=\tilde{p}+\i\tilde{q}$. This way, the weighted least-squares
problem (\ref{eq:wls}) reduces to 
\begin{align}
 & \min_{(\theta,u)\in\R^{2n}}\|s(u\odot\exp(\i\theta))-\tilde{s}\|_{d}^{2}\quad\text{ s.t. }u=\tilde{u}\nonumber \\
= & \min_{\theta\in\R^{n}}\|s(\tilde{u}\odot\exp(\i\theta))-\tilde{s}\|_{d}^{2}\nonumber \\
= & \min_{x\in\C_{1}^{n}}\|s(\tilde{u}\odot x)-\tilde{s}\|_{d}^{2}\label{eq:wls_angle}
\end{align}
where the complex unit $n$-torus is denoted
\begin{equation}
\C_{1}^{n}=\left\{ x\in\C^{n}:|x_{i}|=1,\ \forall i=1,\dots,n\right\} .\label{eq:Cone}
\end{equation}
The reduced problem (\ref{eq:wls_angle}) fixes $u=\tilde{u}$ and
seeks to recover the angles $\theta_{\gnd}$ encoded by the ground
truth phase vector 
\begin{equation}
x_{\gnd}\coloneqq\exp(\i\theta_{\gnd})\label{eq:zdef}
\end{equation}
from noisy complex power measurements $\tilde{s}\approx s(\tilde{u}\odot x_{\gnd})$. 

In Section~\ref{subsec:deriv}, we revisit the reformulation of \cite{guzel2024power}
using a more abstract notation that exposes the underlying algebraic
structure. This yields a streamlined derivation of the equivalence
between (\ref{eq:wls_angle}) and the homogeneous quadratic problem
\begin{align}
x_{\wls}\coloneqq\exp(\i\theta_{\wls}) & =\arg\min_{x\in\C_{1}^{n}}\norm{Cx}^{2}{\tag{P}},\label{eq:P}
\end{align}
hence identifying (\ref{eq:wls}) under Assumption~\ref{asm} as
a special case of phase synchronization \cite[Eqn.~13]{singer2011angular}\cite{boumal2016nonconvex,bandeira2017tightness}
\[
\max_{x\in\C^{n}}\quad x^{*}Hx\quad\text{ s.t. }|x_{1}|=\cdots=|x_{n}|=1,
\]
with $H=-C^{*}C$. Problem (\ref{eq:P}) serves as the starting point
for the spectral methods analyzed in the next two sections. The practical
scope of Assumption~\ref{asm} is discussed in Section~\ref{subsec:assumptions}.

\subsection{Reduction to phase synchronization}\label{subsec:deriv}

We begin by introducing notation that makes the algebraic structure
of the complex power measurements explicit. The complex power flow
operator can be written as
\begin{equation}
s:\C^{n}\to\C^{m},\quad s(v)=\conj(Av)\odot(Bv)\label{eq:sdef_AB}
\end{equation}
where $A$ maps the list of bus voltages $v$ to the current phasors
underlying our power measurements, whereas $B$ selects the voltage
phasor at the terminal where each power measurement is taken:
\begin{equation}
A=\begin{bmatrix}Y_{\bus}\\
Y_{\From}\\
Y_{\To}
\end{bmatrix}\in\C^{m\times n},\quad B=\begin{bmatrix}I\\
E_{\From}\\
E_{\To}
\end{bmatrix}\in\{0,1\}^{m\times n}.
\end{equation}
Recall from (\ref{eq:sv_def}) that $m=n+2l$ for a system with $n$
buses and $l$ lines. Since each row of $B$ contains exactly one
nonzero entry, $B$ satisfies the following elementary but useful
identity.
\begin{lem}
\label{lem:distribute}Let $B\in\{0,1\}^{m\times n}$ satisfy $B\one_{n}=\one_{m}$.
Then, for any $x,y\in\C^{n}$, $B(x\odot y)=(Bx)\odot(By)$.
\end{lem}
We now use our new notation to derive the reformulation of (\ref{eq:wls})
into (\ref{eq:P}) as first introduced in \cite{guzel2024power}.
\global\long\def\all{\mathrm{all}}%

\begin{prop}
\label{prop:Cmat}Given model matrices $A\in\C^{m\times n}$ and $B\in\{0,1\}^{m\times n}$
with $B\one_{n}=\one_{m}$, let $s$ be as defined in (\ref{eq:sdef_AB}).
For any voltage magnitudes $\tilde{u}\in\R_{+}^{n}$, complex power
measurements $\tilde{s}\in\C^{m}$, and weights $d\in\R_{+}^{m}$,
we have
\begin{gather*}
\|s(\tilde{u}\odot x)-\tilde{s}\|_{d}^{2}=\|Cx\|^{2}\quad\text{for all }x\in\C_{1}^{n},
\end{gather*}
where
\begin{equation}
C\coloneqq\Diag(d)^{1/2}\left[\Diag(B\tilde{u})A\Diag(\tilde{u})-\Diag(\conj(\tilde{s}))B\right].\label{eq:Cdef}
\end{equation}
\end{prop}
\begin{IEEEproof}
Let $D=\Diag(d)$ and write $C=D^{1/2}C_{\all}$ for clarity. Since
$x\in\C_{1}^{n}$, Lemma~\ref{lem:distribute} gives
\[
(Bx)\odot\conj(Bx)=B(x\odot\conj(x))=B\one=\one.
\]
Thus $Bx\in\C_{1}^{m}$, and the diagonal matrix $Q=\Diag(Bx)$ is
unitary. We next rewrite the residual. Using Lemma~\ref{lem:distribute},
\begin{align}
 & \conj(s(\tilde{u}\odot x)-\tilde{s})\nonumber \\
= & \conj(B(\tilde{u}\odot x))\odot(A(\tilde{u}\odot x))-\conj(\tilde{s})\nonumber \\
= & (B\tilde{u})\odot\conj(Bx)\odot(A(\tilde{u}\odot x))-\conj(\tilde{s})\nonumber \\
= & \conj(Bx)\odot[(B\tilde{u})\odot(A(\tilde{u}\odot x))-\conj(\tilde{s})\odot(Bx)]\nonumber \\
= & Q^{*}C_{\all}x.\label{eq:Q*Jx}
\end{align}
The fourth line uses uses $(Bx)\odot\conj(Bx)=\one$ to factor out
$\conj(Bx)$ from the second term. Finally, since $D^{1/2}$ is real
diagonal, it commutes with $Q^{*}$. Moreover, $Q$ is unitary. Therefore,
\begin{align*}
\|s(\tilde{u}\odot x)-\tilde{s}\|_{d}^{2} & =\|D^{1/2}[s(\tilde{u}\odot x)-\tilde{s}]\|^{2}\\
 & =\|D^{1/2}\conj(s(\tilde{u}\odot x)-\tilde{s})\|^{2}\\
 & =\|D^{1/2}Q^{*}C_{\all}x\|^{2}=\|Q^{*}D^{1/2}C_{\all}x\|^{2}\\
 & =\|D^{1/2}C_{\all}x\|^{2}=\|Cx\|^{2}.
\end{align*}
This proves the claim. 
\end{IEEEproof}

\subsection{Practical implications of Assumption~\ref{asm}}\label{subsec:assumptions}

Assumption~\ref{asm} has two distinct modeling implications. First,
$w_{u}\to\infty$ reduces the full state-estimation problem into an
angle estimation subproblem. This is appropriate when the measured
magnitudes are accurate enough that the dominant uncertainty is in
the power-flow measurements. If $\tilde{u}=|v_{\gnd}|$, then the
angle error is driven entirely by the power-measurement error. But
if $\tilde{u}$ is inaccurate, then fixing $u=\tilde{u}$ introduces
a modeling error, and this error sets a floor on the best achievable
accuracy in estimating the angles. 

Second, $w_{p}=w_{q}=d$ requires active and reactive power residuals
to be weighted in matched pairs. Each complex power measurement is
either used through both components $(\tilde{p}_{i},\tilde{q}_{i})$,
or omitted from the reduced problem. If active and reactive components
have different noise levels, or if only one component is available,
then the exact reduction to (\ref{eq:P}) no longer applies directly.
A natural way to address this limitation is to incorporate the spectral
angle update into an end-to-end solver for (\ref{eq:wls}), in which
inaccurate or missing measurements are estimated or updated separately.
We leave this extension as future work.

\section{Spectral Initialization}\label{sec:init}

At the end of the previous section, we showed that the standard Power
System State Estimation formulation (\ref{eq:wls}) under Assumption~\ref{asm}
is equivalent to (\ref{eq:P}). Spectral initialization works by replacing
the individual unit-modulus constraints $|x_{i}|=1$ by the single
spherical constraint $\|x\|^{2}=n$ to yield a spectral relaxation
\begin{align}
\min_{x\in\C_{1}^{n}}\|Cx\|^{2} & =\min_{x\in\C^{n}}\|Cx\|^{2}\quad\text{s.t. }|x_{i}|^{2}=1,\ i=1,\dots,n\nonumber \\
 & \ge\min_{x\in\C^{n}}\|Cx\|^{2}\quad\text{s.t. }\sum_{i=1}^{n}|x_{i}|^{2}=n\nonumber \\
 & =n\lambda_{\min}(C^{*}C).\label{eq:spectral_relax}
\end{align}
The solution is any bottom eigenvector $x_{\min}$ of the matrix $C^{*}C$
scaled to satisfy $\|x_{\min}\|^{2}=n$. To turn the solution of the
relaxation into a feasible point for the original problem (\ref{eq:P}),
we simply project $x_{\min}$ back onto $\C_{1}^{n}$. The Euclidean
projection 
\[
\proj(v)\coloneqq\arg\min_{x\in\C_{1}^{n}}\|v-x\|
\]
has a (possibly nonunique) closed form expression 
\begin{equation}
[\proj(v)]_{i}=\begin{cases}
\frac{v_{i}}{|v_{i}|} & v_{i}\ne0,\\
1 & v_{i}=0.
\end{cases}\label{eq:proj}
\end{equation}
The spectral initialization is therefore 
\begin{equation}
x_{\si}\coloneqq\proj(x_{\min}),\quad v_{\si}\coloneqq\tilde{u}\odot x_{\si},\quad\theta_{\si}\coloneqq\arg x_{\si}.\label{eq:si_def}
\end{equation}
Guzel and Zhang~\cite{guzel2024power} observed experimentally that
$x_{\si}$ is often already a highly accurate solution of the angle
subproblem, and that $v_{\si}$ provides an excellent initialization
for the original (\ref{eq:wls}) problem. We now explain why this
happens.

First consider the noiseless case. Substituting $\epsilon_{s}=0$
and $\epsilon_{u}=0$ into (\ref{eq:noisy_meas}) yields a complete
set of perfect measurements
\[
\tilde{u}=u_{\gnd},\quad\tilde{s}=s_{\gnd}\coloneqq s(u_{\gnd}\odot\exp(\i\theta_{\gnd})).
\]
Define the corresponding noiseless $C$ matrix in (\ref{eq:Cdef}),
explicitly
\begin{multline}
C_{0}\coloneqq\Diag(d)^{1/2}[\Diag(Bu_{\gnd})A\Diag(u_{\gnd})\\
-\Diag(\conj(s_{\gnd}))B].\label{eq:C0def}
\end{multline}
The key observation is that the ground-truth phase vector $x_{\gnd}=\exp(\i\theta_{\gnd})$
lies in the right nullspace of $C_{0}$:
\begin{equation}
C_{0}x_{\gnd}=0.\label{eq:C0_null}
\end{equation}
If additionally $\sigma_{n-1}(C_{0})>0,$ then $C_{0}$ has a one-dimensional
nullspace, spanned by $x_{\gnd}$. In other words, every bottom eigenvector
$x_{\min}$ of $C_{0}^{*}C_{0}$ is a scalar multiple of the ground
truth $x_{\gnd}$, as in $x_{\min}=\alpha x_{\gnd}$ for $\alpha\in\C\backslash\{0\}$.
This is the fundamental reason why spectral initialization (\ref{eq:si_def})
exactly recovers the ground truth, just as Guzel and Zhang~\cite{guzel2024power}
had experimentally observed. But if $\sigma_{n-1}(C_{0})=0$, then
$C_{0}^{*}C_{0}$ also has bottom eigenvectors orthogonal to $x_{\gnd}$,
so spectral initialization can fail to recover $x_{\gnd}$, even with
noiseless measurements. 

In the general noisy case with $\epsilon_{s}\ne0$ and $\epsilon_{u}\ne0$,
the corresponding $C$ matrix in (\ref{eq:Cdef}) is perturbed from
the noiseless $C_{0}$ matrix defined above in (\ref{eq:C0def}).
Nevertheless, if this perturbation is small relative to $\sigma_{n-1}(C_{0})$,
then the bottom eigenspace of $C^{*}C$ remains close to the nullspace
of $C_{0}$, so the spectral initializer $x_{\si}$ remains close
to $x_{\gnd}$.
\begin{lem}
\label{lem:si}For $C,C_{0}\in\C^{m\times n}$, let $x_{\gnd}\in\C^{n}$
satisfy $C_{0}x_{\gnd}=0$, and let $x_{\si}=\proj(x_{\min})$, where
$x_{\min}$ is a bottom eigenvector of $C^{*}C$. If $\sigma_{n-1}(C_{0})>0$
and
\begin{equation}
\|C-C_{0}\|_{\op}\le\frac{1}{3}\sigma_{n-1}(C_{0}),\label{eq:main_mle_assump}
\end{equation}
then 
\begin{equation}
\dist(x_{\si},x_{\gnd})\le8\sqrt{2}\,\frac{\|(C-C_{0})x_{\gnd}\|}{\sigma_{n-1}(C_{0})}.\label{eq:main_mle}
\end{equation}
\end{lem}
We defer the proof to Section~\ref{subsec:specinit}. We conclude
therefore that $\sigma_{n-1}(C_{0})$ plays the role of an angle-observability
margin, determining both the size of a tolerable perturbation in $C$,
as well as the accuracy of the resulting estimation $x_{\si}$. In
fact, under Assumption~\ref{asm}, it exactly coincides with the
classical observability margin from Section~\ref{subsec:observe}.
\begin{lem}
\label{lem:C0_observability} Let $\rho$ denote the observability
margin defined in (\ref{eq:observability_margin}). Under Assumption~\ref{asm},
the matrix $C_{0}$ in (\ref{eq:C0def}) satisfies $\sigma_{n-1}(C_{0})=\rho.$
\end{lem}
We defer the proof to Section~\ref{subsec:proof_observ}. The same
quantity therefore has two interpretations. In the classical WLS theory,
$\rho>0$ guarantees local identifiability near the ground truth.
In the noiseless angle-estimation problem, the same condition guarantees
exact recovery by the global spectral relaxation, without requiring
an initial guess.

It remains to combine the previous three lemmas and express the accuracy
of the estimated angles in power-system interpretable quantities.
\begin{thm}[Spectral initialization]
\label{thm:main_si}Under Assumption~\ref{asm}, denote $\rho$
as the observability margin in (\ref{eq:observability_margin}), and
define the normalized noise level 
\begin{equation}
\eta\coloneqq\frac{K_{s}\|\epsilon_{s}\|_{\infty,\obs}+K_{u}\|\epsilon_{u}\|_{\infty}}{\rho}\label{eq:eta_def}
\end{equation}
where $\|\epsilon_{s}\|_{\infty,\obs}\coloneqq\max\{|(\epsilon_{s})_{i}|:d_{i}>0\}$
and
\begin{gather*}
K_{s}=K_{0}\|B\|_{\op},\qquad K_{u}=3K_{0}\|A\|_{\op}\|u_{\gnd}\|_{\infty},\\
K_{0}=4\sqrt{2}\pi\|d\|_{\infty}^{1/2}.
\end{gather*}
If $\eta\le\frac{4\sqrt{2}}{3}\pi$ and $\|\epsilon_{u}\|_{\infty}\le\|u_{\gnd}\|_{\infty}$,
then the spectral initialization $\theta_{\si}$ defined in (\ref{eq:si_def})
satisfies
\begin{equation}
\rmse(\theta_{\si},\theta_{\gnd})\le\eta\text{ radians}.\label{eq:si}
\end{equation}
\end{thm}
\begin{IEEEproof}
Write $E\coloneqq C-C_{0}$. First, substituting (\ref{eq:C0_null})
and Lemma~\ref{lem:C0_observability} into (\ref{eq:main_mle}) yields
\begin{equation}
\frac{1}{\sqrt{n}}\dist(x_{\si},x_{\gnd})\le8\sqrt{2}\,\frac{\|Ex_{\gnd}\|}{\rho\sqrt{n}}\le8\sqrt{2}\,\frac{\|E\|_{\op}}{\rho}\label{eq:si_rms}
\end{equation}
since $\|Ex_{\gnd}\|\le\|E\|_{\op}\|x_{\gnd}\|$ and $\|x_{\gnd}\|=\sqrt{n}.$
Next, we write $E=E_{s}+E_{u}$ explicitly where \begin{subequations}\label{eq:Edef}
\begin{align}
E_{s} & \coloneqq-D^{1/2}\Diag(\conj(\epsilon_{s}))B,\\
E_{u} & \coloneqq D^{1/2}[\Diag(B\epsilon_{u})A\Diag(\tilde{u})\nonumber \\
 & \qquad\qquad+\Diag(Bu_{\gnd})A\Diag(\epsilon_{u})].
\end{align}
\end{subequations}Substituting $\|E\|_{\op}\le\|E_{s}\|_{\op}+\|E_{u}\|_{\op}$
into (\ref{eq:si_rms}) and (\ref{eq:main_mle_assump}), and translating
phasor distance to angular error using (\ref{eq:rmse_dist}) completes
the proof. 
\end{IEEEproof}
Hence, for sufficiently small normalized noise level $\eta$, spectral
initialization recovers the ground truth to $\eta$ accuracy. In particular,
if $\eta=0$, then it recovers the ground truth in one shot. Otherwise,
for larger values of $\eta$, spectral initialization also provides
a reliable starting point for iterative refinement to global optimality. 
\begin{thm}[Weighted least-squares]
\label{thm:wls}Under the same conditions as Theorem~\ref{thm:main_si},
the globally optimal solution $\theta_{\wls}$ in (\ref{eq:P}) satisfies
\begin{align}
\rmse(\theta_{\wls},\theta_{\gnd}) & \le\frac{1}{2}\eta\text{ radians},\label{eq:mle_gnd}\\
\rmse(\theta_{\si},\theta_{\wls}) & \le\frac{3}{2}\eta\text{ radians}.\label{eq:si_mle}
\end{align}
\end{thm}
We defer the proof to Section~\ref{subsec:specinit}. The error bound
(\ref{eq:si_mle}) places the spectral initializer near the global
minimizer, while the error bound (\ref{eq:mle_gnd}) suggests a potential
improvement in accuracy by further refinement. 

\section{Spectral certification}\label{sec:cert}

Once local refinement returns a candidate angle estimate, the remaining
question is how to verify, after refinement, that the recovered point
is indeed the global solution of (\ref{eq:P}). The phase synchronization
formulation provides a computable \emph{a posteriori} certificate
for this purpose. For $x\in\C_{1}^{n}$, define 
\begin{align}
S(x) & \coloneqq C^{*}C-\Re\!\left[\ddiag(C^{*}Cxx^{*})\right].\label{eq:S_def}
\end{align}
The following is the standard dual certificate for phase synchronization.
\begin{prop}[{\cite[Lem.~2]{boumal2016nonconvex}}]
\label{prop:cert} Given $C\in\C^{m\times n}$, define the objective
$f(x)\coloneqq\|Cx\|^{2}$ and let $x_{\wls}\coloneqq\arg\min_{x\in\C_{1}^{n}}f(x)$
be globally optimal. For any $x\in\C_{1}^{n}$, the optimality gap
at $x$ is bounded as
\begin{equation}
0\;\le\;f(x)-f(x_{\wls})\;\le\;\delta(x)\coloneqq-n\lambda_{\min}(S(x)),\label{eq:delta_def}
\end{equation}
In particular, if $\delta(x)=0$, then $x$ is global optimal. 
\end{prop}
After any algorithm returns an angle estimate $x$, one can compute
$\lambda_{\min}(S(x))$ and certify an objective gap. If $\delta(x)\approx0$,
then $x$ is near globally optimal for (\ref{eq:P}). 

Nevertheless, Proposition~\ref{prop:cert} has two critical limitations.
First, the certificate is conservative. A small value of $\delta(x)\approx0$
certifies near-global optimality $f(x)\approx f(x_{\wls})$, but a
large value indicates only a failure to certify. In particular, failure
may occur even with a globally optimal $x=x_{\wls}$. Second, certified
global optimality does not imply recovery. If the measurements are
uninformative or the noise is too large relative to observability,
then a certifiably globally optimal $x=x_{\wls}$ may still be far
from $x_{\gnd}$. 

For the certificate to be useful as a post hoc verification tool,
it should be exact, in that it should certify whenever it is possible
to do so, and it should directly certify recovery, meaning that $\delta(x)\approx0$
should indicate that $\dist(x,x_{\gnd})\approx0$. The experiments
in~\cite{guzel2024power} suggested that both properties hold below
an apparent noise threshold. Our theory below explains this behavior.
\begin{lem}
\label{lem:main_sdp}Let $C_{0}\in\C^{m\times n}$ satisfy $C_{0}x_{\gnd}=0$
for some $x_{\gnd}\in\C^{n}$. Under the same conditions as Proposition~\ref{prop:cert},
if 
\[
\opnorm{C-C_{0}}\le\frac{\sigma_{n-1}(C_{0})}{2},\;\norm{(C-C_{0})x_{\gnd}}\le\frac{\sigma_{n-1}(C_{0})}{4(\kappa_{\infty}+1)},
\]
where $\kappa_{\infty}\coloneqq\frac{\max_{i}\norm{C_{0}e_{i}}}{\sigma_{n-1}(C_{0})}$,
then $x_{\wls}$ is the unique global solution (up to absolute phase),
and the function $\delta$ in (\ref{eq:delta_def}) satisfies $\delta(x)=0$
if and only if $\dist(x,x_{\wls})=0$. 
\end{lem}
Translating this result into power-system interpretable quantities
by repeating the proof of Theorem~\ref{thm:main_si} yields the following.
\begin{thm}[Spectral certification]
\label{thm:cert}Under the same conditions as Theorem~\ref{thm:main_si},
if the normalized noise level satisfies
\[
\eta\le\frac{1}{4\sqrt{n}(\kappa_{\infty}+1)},
\]
then the solution $\theta_{\wls}$ in (\ref{eq:P}) is unique, and
$\delta(\exp(\i\theta))=0$ holds if and only if $\rmse(\theta,\theta_{\wls})=0$. 
\end{thm}
If the normalized noise level $\eta$ is low, then spectral certification
is exact, the WLS estimate is unique up to phase, and Theorem~\ref{thm:wls}
shows that this estimate is close to the true phase angles. In particular,
if $\eta=0$, then spectral initialization exactly recovers the ground
truth $\theta_{\gnd}$, and the spectral certification in turn certifies
recovery of the ground truth. The actual threshold in Theorem~\ref{thm:cert}
is quite conservative, however; in our experiments below, we find
that certification remains reliable at noise levels far beyond what
this sufficient condition predicts.

\section{Experiments}\label{sec:Experiments}

\global\long\def\emp{\mathrm{emp}}%
\global\long\def\cert{\mathrm{cert}}%
\global\long\def\thresh{\mathrm{thresh}}%
\global\long\def\xsf{\mathsf{x}}%
\global\long\def\ysf{\mathsf{y}}%

\begin{figure}[t]
\subfloat[]{\centering\includegraphics[width=1\columnwidth]{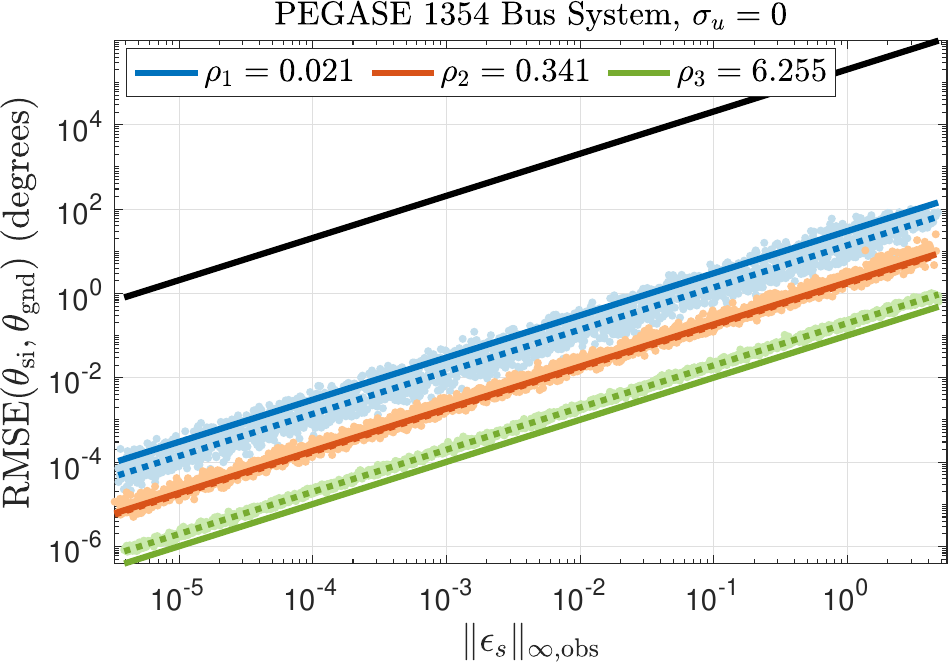}

}

\subfloat[]{\centering\includegraphics[width=1\columnwidth]{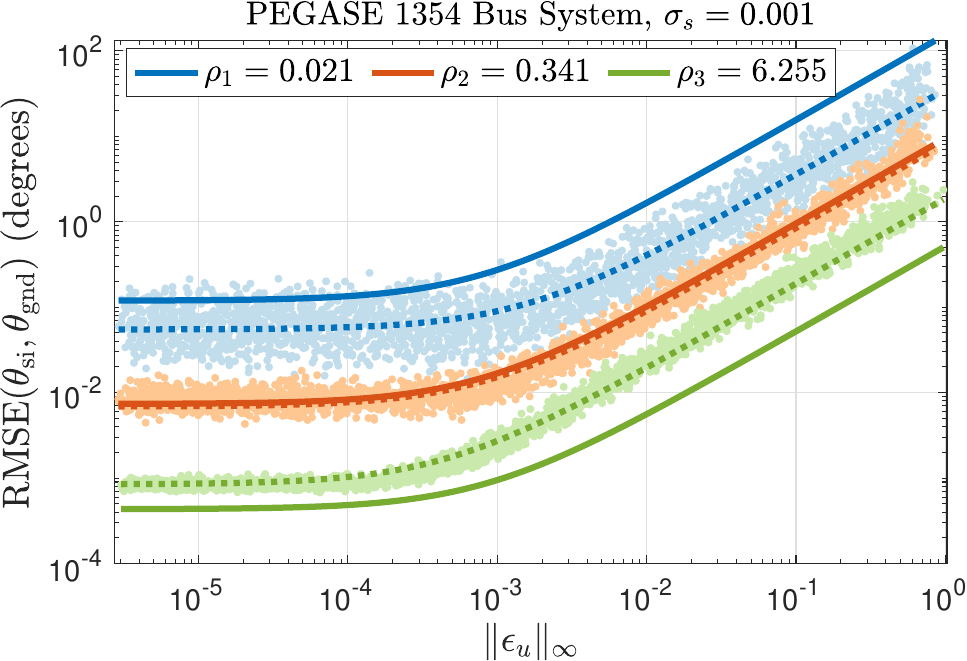}

}\caption{Spectral initialization error as a function of: (a) $\|\epsilon_{s}\|_{\infty,\protect\obs}$
with $\epsilon_{u}=0$; (b) $\|\epsilon_{u}\|_{\infty}$ with fixed
$\|\epsilon_{s}\|_{\infty,\protect\obs}>0$. Each color corresponds
to a different observability margin $\rho$. The solid color lines
show fitted curves for the standard model (\ref{eq:tilde_eta_def}),
while the dashed color lines show fitted curves for the counterfactual
model (\ref{eq:eq:tilde_eta_def2}). Markers show the measured RMSE
values. The solid black line shows the upper bound from Theorem~\ref{thm:main_si}.}
\label{fig:eta}
\end{figure}
\begin{table*}
\caption{Theoretical and fitted constants for the spectral-initialization error
model up to 4 significant digits. The $R^{2}$ values are computed
on the log scale over the low-error regime $\protect\ysf_{j}\le3^{\circ}$.}\label{tab:KsKu}

\centering%
\begin{tabular}{|c|cc|ccc|ccccc|}
\hline 
\multirow{2}{*}{Case} & \multicolumn{2}{c|}{Theoretical Constants} & \multicolumn{3}{c|}{Standard Model (\ref{eq:tilde_eta_def})} & \multicolumn{5}{c|}{Counterfactual Model (\ref{eq:eq:tilde_eta_def2})}\tabularnewline
\cline{2-11}
 & $K_{s}$ & $K_{u}$ & $\tilde{K}_{s}$ & $\tilde{K_{u}}$ & $R^{2}$ & $\tilde{K}_{s}$ & $\tilde{K_{u}}$ & $\beta$ & $\xi$ & $R^{2}$\tabularnewline
\hline 
case1354pegase & 7.540E+01 & 1.523E+06 & 1.103E-02 & 5.646E-02 & 0.9337 & 1.354E-02 & 0.08432 & 0.7464 & 0.5140 & 0.9841\tabularnewline
case1888rte & 7.540E+01 & 4.545E+06 & 1.107E-02 & 2.604E-02 & 0.9387 & 1.950E-02 & 5.014E-02 & 0.5959 & 0.5524 & 0.9844\tabularnewline
case1951rte & 7.540E+01 & 4.633E+06 & 1.095E-02 & 2.722E-02 & 0.9448 & 1.942E-02 & 4.911E-02 & 0.6255 & 0.6305 & 0.9844\tabularnewline
case2383wp & 5.620E+01 & 1.659E+06 & 7.346E-03 & 7.507E-03 & 0.7910 & 1.771E-02 & 4.907E-02 & 0.6491 & 0.3166 & 0.9594\tabularnewline
case2736sp & 5.620E+01 & 2.681E+06 & 7.258E-03 & 4.678E-03 & 0.8442 & 1.504E-02 & 2.376E-02 & 0.7015 & 0.4018 & 0.9639\tabularnewline
case2737sop & 5.620E+01 & 2.668E+06 & 7.756E-03 & 3.705E-03 & 0.8663 & 1.500E-02 & 1.698E-02 & 0.7272 & 0.4372 & 0.9637\tabularnewline
case2746wop & 5.620E+01 & 8.041E+06 & 7.734E-03 & 5.055E-03 & 0.8598 & 1.507E-02 & 2.326E-02 & 0.7188 & 0.4242 & 0.9664\tabularnewline
case2746wp & 5.620E+01 & 8.021E+06 & 7.445E-03 & 6.723E-03 & 0.8561 & 1.527E-02 & 2.897E-02 & 0.7015 & 0.4482 & 0.9621\tabularnewline
case2848rte & 7.540E+01 & 4.614E+06 & 9.682E-03 & 1.004E-02 & 0.9139 & 1.897E-02 & 2.665E-02 & 0.6717 & 0.5542 & 0.9779\tabularnewline
case2868rte & 7.540E+01 & 4.61E+06 & 1.018E-02 & 1.397E-02 & 0.9296 & 1.897E-02 & 3.433E-02 & 0.6704 & 0.5456 & 0.9835\tabularnewline
case2869pegase & 7.540E+01 & 1.913E+06 & 5.574E-03 & 2.160E-02 & 0.9007 & 7.707E-03 & 6.235E-02 & 0.8388 & 0.5258 & 0.9678\tabularnewline
case3012wp & 5.894E+01 & 2.904E+06 & 7.066E-03 & 7.376E-03 & 0.8674 & 1.411E-02 & 2.866E-02 & 0.7125 & 0.4799 & 0.9648\tabularnewline
case3120sp & 5.620E+01 & 2.872E+06 & 7.085E-03 & 5.949E-03 & 0.8285 & 1.464E-02 & 3.045E-02 & 0.6975 & 0.3781 & 0.9586\tabularnewline
case3375wp & 6.649E+01 & 3.874E+06 & 7.462E-03 & 1.072E-02 & 0.8762 & 1.556E-02 & 3.152E-02 & 0.6856 & 0.5602 & 0.9623\tabularnewline
case6468rte & 7.540E+01 & 3.925E+06 & 6.089E-03 & 6.217E-03 & 0.8865 & 1.495E-02 & 3.117E-02 & 0.6772 & 0.4425 & 0.9648\tabularnewline
case6470rte & 7.540E+01 & 3.967E+06 & 5.861E-03 & 6.354E-03 & 0.8825 & 1.455E-02 & 3.477E-02 & 0.6992 & 0.4566 & 0.9738\tabularnewline
case6495rte & 7.540E+01 & 3.943E+06 & 6.154E-03 & 7.062E-03 & 0.8837 & 1.531E-02 & 3.732E-02 & 0.6968 & 0.4670 & 0.9760\tabularnewline
case6515rte & 7.540E+01 & 3.945E+06 & 6.308E-03 & 7.374E-03 & 0.8848 & 1.450E-02 & 3.913E-02 & 0.7214 & 0.4657 & 0.9730\tabularnewline
case9241pegase & 1.218E+02 & 3.139E+06 & 3.733E-03 & 9.445E-03 & 0.8732 & 6.000E-03 & 4.902E-02 & 0.7661 & 0.4146 & 0.9716\tabularnewline
case13659pegase & 1.218E+02 & 3.151E+06 & 3.401E-03 & 4.530E-03 & 0.9089 & 1.566E-02 & 2.849E-02 & 0.6134 & 0.5315 & 0.9658\tabularnewline
\hline 
\end{tabular}
\end{table*}

Theorems~\ref{thm:main_si}, \ref{thm:wls}, and~\ref{thm:cert}
provide deterministic worst-case guarantees and are therefore expected
to be conservative. This section quantifies that conservatism empirically.
To do this, we fit an empirical noise level
\begin{equation}
\tilde{\eta}\coloneqq\frac{\tilde{K}_{s}\|\epsilon_{s}\|_{\infty,\mathrm{obs}}+\tilde{K}_{u}\|\epsilon_{u}\|_{\infty}}{\rho},\label{eq:tilde_eta_def}
\end{equation}
which has the same form as the theoretical noise level $\eta$ in
Theorem~\ref{thm:main_si}, but with constants $\tilde{K}_{s},\tilde{K}_{u}$
fitted from data. We then test two empirical predictions. First, the
spectral initializer should satisfy 
\[
\rmse(\theta_{\si},\theta_{\gnd})\approx\tilde{\eta}\qquad\text{for small }\tilde{\eta}.
\]
Second, after Gauss--Newton refinement initialized at $\theta_{\si}$,
the refined point $\theta$ should satisfy 
\[
\delta(\exp(i\theta))\approx0\qquad\text{for small }\tilde{\eta}.
\]
These are the empirical counterparts of Theorems~\ref{thm:main_si},
\ref{thm:wls}, and~\ref{thm:cert}.

Our experiments are performed on the MATPOWER test systems~\cite{zimmerman2010matpower}.
For each test system, we fix the network matrices $Y_{\bus},Y_{\From},Y_{\To},F_{\From},F_{\To}$
and the ground-truth voltage phasor $v_{\gnd}=u_{\gnd}\odot\exp(i\theta_{\gnd})$.
We use binary measurement weights $d\in\{0,1\}^{m}$, and vary the
observability margin $\rho$ by changing the observed measurement
mask $d$. Masks with $\rho=0$ are discarded, leaving $100$ observable
measurement masks. For each mask, we generate zero-mean i.i.d. Gaussian
perturbation directions for $\epsilon_{s}$ and $\epsilon_{u}$ with
variances $\sigma_{s}^{2}$ and $\sigma_{u}^{2}$, and then rescale
them if necessary to achieve the prescribed values of $\|\epsilon_{s}\|_{\infty,\mathrm{obs}}$
and $\|\epsilon_{u}\|_{\infty}$. Angular errors are computed modulo
the global phase ambiguity using (\ref{eq:rmsedef}). The experiments
are performed in MATLAB 2024b with random seed \texttt{rng(0)}, and
each plotted point represents one trial over a range of noise realizations.

\subsection{Accuracy of the spectral initializer}\label{subsec:exp1}

We fit (\ref{eq:tilde_eta_def}) in the low-error regime, retaining
samples with $\rmse(\theta_{\si,j},\theta_{\gnd})\le3^{\circ}.$ Define
$\xsf_{s,j}\coloneqq\|\epsilon_{s,j}\|_{\infty,\mathrm{obs}},$ $\xsf_{u,j}\coloneqq\|\epsilon_{u,j}\|_{\infty},$
$\ysf_{j}\coloneqq\rmse(\theta_{\si,j},\theta_{\gnd})$, and $\mathcal{I}\coloneqq\{j:\ysf_{j}\le3^{\circ}\}$.
We use MATLAB's \texttt{lsqnonlin} to compute $\tilde{K}_{s},\tilde{K}_{u}$
by the log-scale regression 
\[
\min_{K_{s},K_{u}\ge0}\sum_{j\in\mathcal{I}}\left(\log\ysf_{j}-\log\frac{K_{s}\xsf_{s,j}+K_{u}\xsf_{u,j}}{\rho_{j}}\right)^{2}.
\]
The solid lines in Fig.~\ref{fig:eta} show the resulting fit for
the PEGASE $1354$-bus test case. As expected, the empirical dependence
on $\|\epsilon_{s}\|_{\infty,\mathrm{obs}}$ and $\|\epsilon_{u}\|_{\infty}$
is approximately linear, but the leading constants are significantly smaller
than those prediced by theory. The empirically extracted $\tilde{K}_{s}=0.011$
and $\tilde{K}_{u}=0.057$ (with fit quality $R^{2}=0.93$) are orders
of magnitude smaller than their corresponding theoretical values $K_{s}=75.4$
and $K_{u}=1.52\cdot10^{6}$, so spectral initialization is substantially
more accurate in practice than worst-case theory predicts. Table~\ref{tab:KsKu}
confirms the same trend across the other test systems.

The remaining mismatch between theory and data is mainly in the dependence
on observability $\rho$. As a counterfactual test, we repeat the
same fit procedure for (\ref{eq:tilde_eta_def}), but allow the exponents
on $\rho$ to vary: 
\begin{equation}
\tilde{\eta}\coloneqq\tilde{K}_{s}\frac{\|\epsilon_{s}\|_{\infty,\mathrm{obs}}}{\rho^{\beta}}+\tilde{K}_{u}\frac{\|\epsilon_{u}\|_{\infty}}{\rho^{\xi}}.\label{eq:eq:tilde_eta_def2}
\end{equation}
As seen in Fig.~\ref{fig:eta}, the flexible model gives a much better
fit for the PEGASE $1354$-bus case, boosting the fit quality to $R^{2}=0.98$.
The fitted exponents $\beta=0.746$ and $\xi=0.514$ are well below
one. This suggests that spectral initialization is much more robust
to reductions in observability $\rho\to0^{+}$ than worst-case theory
predicts. Table~\ref{tab:KsKu} confirms these same pattern across
other test cases.

\subsection{Certification after local refinement}\label{subsec:exp2}

Theorem~\ref{thm:cert} and Theorem~\ref{thm:wls} gives a sufficient
condition under which Gauss--Newton initialized at $\theta_{\si}$
produces a point that can be certified as near-global and near-recovering.
In the empirical scale (\ref{eq:tilde_eta_def}), this threshold becomes
\begin{equation}
\tilde{\eta}\le\tilde{\eta}_{\cert}\coloneqq\frac{\min\left\{ \tilde{K}_{s}/K_{s},\tilde{K}_{u}/K_{u}\right\} }{4\sqrt{n}(\kappa_{\infty}+1)}.\label{eq:eta_cert}
\end{equation}
To test the conservatism of (\ref{eq:eta_cert}), we use the full
set of complex bus power measurements, take $\sigma_{u}\in\{0,0.0001,0.001,0.01\}$,
sweep $\sigma_{s}$ over $100$ logarithmically spaced values in $[10^{-5},1]$,
and run $50$ random trials for each $(\sigma_{s},\sigma_{u})$ pair.
For each trial, we compute the spectral estimate $\theta_{\si}$,
perform 20 iterations of Gauss--Newton refinement, and evaluate the
certificate. We count a trial as certified when $\delta(\exp(i\theta))\le10^{-3}.$

Fig.~\ref{fig:delta} plots the certified suboptimality $\delta(\exp(i\theta))$
against the noise level $\tilde{\eta}$ for the PEGASE $1354$-bus
case. With perfect magnitude measurements ($\sigma_{u}=0$) certification
succeeds up to comparatively high noise levels of $\tilde{\eta}_{\emp}=6.81\cdot10^{-2}$
radians. But when magnitude measurements are also noisy ($\sigma_{u}>0$),
certification begins to fail much earlier, at noise levels of $\tilde{\eta}_{\emp}=6.65\cdot10^{-5}$
radians. Both empirical thresholds remain far above the theoretical
threshold (\ref{eq:eta_cert}), which evaluates to $\tilde{\eta}_{\cert}=1.237\cdot10^{-12}$.
Table~\ref{tab:thresh} shows the same trends across our other test
cases.

\begin{figure}
\centering \includegraphics[width=1\columnwidth]{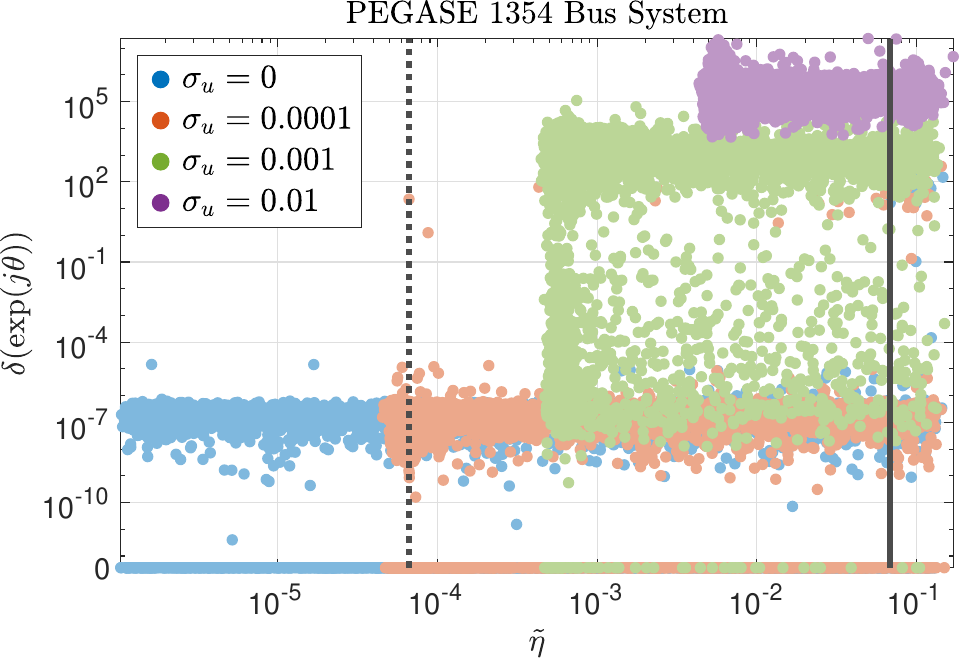}
\caption{Certified suboptimality $\delta(\exp(i\theta))$ after 20 steps of
Gauss--Newton refinement at the spectral estimate. The solid and
dashed vertical lines indicate the largest tested values of $\widetilde{\eta}$
for which $\delta(\exp(i\theta))\le10^{-3}$, with $\epsilon_{u}=0$
and over all samples, respectively.}
\label{fig:delta} 
\end{figure}

\begin{table}
\caption{Theory-implied and empirical certification thresholds up to four
significant digits. The empirical threshold $\tilde{\eta}_{\protect\emp}$
is the largest tested value of $\tilde{\eta}$ for which $\delta(\exp(i\theta))\le10^{-3}$.}\label{tab:thresh}

\centering %
\begin{tabular}{c|ccc}
\hline 
Case & $\tilde{\eta}_{\cert}$  & $\tilde{\eta}_{\emp}$$\ (\epsilon_{u}=0)$ & $\tilde{\eta}_{\emp}$$\ (\epsilon_{u}\neq0)$\tabularnewline
\hline 
case1354pegase & 1.237E-14 & 6.819E-02 & 6.650E-05\tabularnewline
case1888rte & 6.609E-16 & 2.936E-02 & 3.572E-05\tabularnewline
case1951rte & 5.721E-16 & 2.989E-02 & 3.644E-05\tabularnewline
case2383wp & 1.436E-15 & 6.157E-02 & 3.234E-05\tabularnewline
case2736sp & 3.412E-16 & 4.266E-02 & 2.092E-05\tabularnewline
case2737sop & 2.693E-16 & 5.541E-02 & 1.768E-05\tabularnewline
case2746wop & 3.857E-17 & 5.853E-02 & 2.169E-05\tabularnewline
case2746wp & 5.216E-17 & 5.954E-02 & 2.928E-05\tabularnewline
case2848rte & 2.061E-16 & 2.321E-02 & 2.315E-05\tabularnewline
case2868rte & 2.452E-16 & 3.349E-02 & 2.974E-05\tabularnewline
case2869pegase & 1.982E-15 & 6.140E-02 & 1.368E-04\tabularnewline
case3012wp & 3.132E-16 & 5.567E-02 & 3.024E-05\tabularnewline
case3120sp & 2.602E-16 & 6.550E-02 & 2.254E-05\tabularnewline
case3375wp & 2.720E-16 & 5.744E-02 & 3.537E-05\tabularnewline
case6468rte & 9.800E-17 & 3.433E-02 & 2.632E-05\tabularnewline
case6470rte & 9.217E-17 & 3.006E-02 & 2.745E-05\tabularnewline
case6495rte & 1.022E-16 & 3.728E-02 & 2.880E-05\tabularnewline
case6515rte & 1.075E-16 & 4.153E-02 & 2.966E-05\tabularnewline
case9241pegase & 2.272E-16 & 5.776E-02 & 2.854E-04\tabularnewline
case13659pegase & 9.022E-17 & 5.199E-02 & 2.178E-04\tabularnewline
\hline 
\end{tabular}
\end{table}

\section{Conclusions}

This paper gives a theoretical explanation for the empirical effectiveness
of spectral initialization and spectral certification observed in
prior work~\cite{guzel2024power}. When the normalized noise level
$\eta$ defined in Theorem~\ref{thm:main_si} is sufficiently small,
the spectral initializer is $O(\eta)$ accurate, and the certificate
succeeds after local refinement. A central and surprising finding
is that these global spectral methods, which do not require an initial
point, can be characterized in terms of the classical observability
margin $\rho$, even though $\rho$ is traditionally a local quantity
defined near the ground truth.

Experiments find that our theoretical results gives a valid but pessimistic
threshold for the low-noise regime $\eta\approx0$. In practice, spectral
methods continue to perform well far beyond these thresholds. The
gap between theory and practice is unavoidable, because theory must
account for deterministic worst-case perturbations aligned with the
least-observable modes, whereas the random average-case perturbations
used in practice are typically far less damaging. Still, the gap could
be made smaller by sharper bookkeeping or stronger analysis techniques.
While this paper proves that spectral methods work, precisely characterizing
their fundamental limits remains an open question.

Finally, reducing PSSE to an angle-estimation problem where spectral
methods apply relies on the restrictive Assumption~\ref{asm}. Extending
these ideas to settings that relax or remove this assumption is another
important direction for future work.

\appendix

\section{Proofs}

\subsection{Connection with observability (Lemma~\ref{lem:C0_observability})}\label{subsec:proof_observ}

Under Assumption~\ref{asm}, write the magnitude weights as $w_{u}=\beta\one$
with $\beta\to\infty$, and use the power weights $w_{p}=w_{q}=d$
to write $D=\Diag(d)$. The weighted Jacobian at the ground truth
is 
\[
H\coloneqq H(\theta_{\gnd},u_{\gnd})=\begin{bmatrix}\Re\{J_{\theta}\} & \Re\{J_{u}\}\\
\Im\{J_{\theta}\} & \Im\{J_{u}\}\\
0 & \sqrt{\beta}\,I
\end{bmatrix},
\]
where the power-measurement part is 
\[
J_{\theta}=D^{1/2}\frac{\partial}{\partial\theta}s(u\odot\exp(\i\theta)),\quad J_{u}=D^{1/2}\frac{\partial s(u\odot\exp(\i\theta))}{\partial u}.
\]
The Jacobian is singular, with $H\begin{bmatrix}\one_{n}\\
0
\end{bmatrix}=0,$ due to the absolute phase. Its next smallest singular value satisfies
\[
\sigma_{2n-1}(H)=\min_{\substack{(\dot{\theta},\dot{u})\in\R^{2n}\\
\one^{T}\dot{\theta}=0,\ \|(\dot{\theta},\dot{u})\|=1
}
}\sqrt{\|J_{\theta}\dot{\theta}+J_{u}\dot{u}\|^{2}+\beta\|\dot{u}\|^{2}}.
\]
Taking the limit $\beta\to\infty$, we have 
\[
\rho=\lim_{\beta\to\infty}\sigma_{2n-1}(H)=\min_{\substack{\dot{\theta}\in\R^{n}\\
\one^{T}\dot{\theta}=0,\ \|\dot{\theta}\|=1
}
}\|J_{\theta}\dot{\theta}\|=\sigma_{n-1}(J_{\theta}).
\]
It remains to identify $J_{\theta}$. For arbitrary $\dot{\theta}\in\R^{n}$,
write $x(t)=x_{\gnd}\odot\exp(\i t\dot{\theta}),$ and note that
\begin{align*}
\frac{\mathrm{d}}{\mathrm{d}t}s(u_{\gnd}\odot x(t))\bigg|_{t=0} & =\frac{\partial s(u_{\gnd}\odot x_{\gnd})}{\partial\theta}\frac{\mathrm{d}}{\mathrm{d}t}x(t)\bigg|_{t=0}\\
 & =J_{\theta}\left(x(0)\odot j\dot{\theta}\right)=J_{\theta}P\dot{\theta},
\end{align*}
where $P=\Diag(jx_{\gnd})$ is unitary. Also, the identity (\ref{eq:Q*Jx})
in the proof of Proposition~\ref{prop:Cmat} gives 
\[
D^{1/2}\conj\left[s(u_{\gnd}\odot x(t))-s_{\gnd}\right]=Q(x(t))^{*}C_{0}x(t),
\]
where $Q(x)=\Diag(Bx)$ is also unitary. Combining these two identities,
and recalling that $\dot{\theta}$ is real-valued, yields
\begin{align*}
\conj(J_{\theta}P)\dot{\theta} & =\frac{\mathrm{d}}{\mathrm{d}t}D^{1/2}\conj\left[s(u_{\gnd}\odot x(t))-s_{\gnd}\right]\bigg|_{t=0}\\
 & =\frac{\mathrm{d}}{\mathrm{d}t}Q(x(t))^{*}C_{0}x(t)\bigg|_{t=0}\\
 & =Q(x(0))^{*}C_{0}\frac{\mathrm{d}}{\mathrm{d}t}x(t)\bigg|_{t=0}+\frac{\mathrm{d}}{\mathrm{d}t}Q(x(t))^{*}\bigg|_{t=0}C_{0}x(0)\\
 & =Q(x(0))^{*}C_{0}(\i x(0)\odot\dot{\theta})+0\\
 & =\left[Q(x_{\gnd})^{*}C_{0}P\right]\dot{\theta}.
\end{align*}
Since $Q(x_{\gnd})$ and $P$ are both unitary, and singular values
are unchanged under complex conjugation, we have 
\[
\rho=\lim_{t\to\infty}\sigma_{2n-1}(H)=\sigma_{n-1}(J_{\theta})=\sigma_{n-1}(C_{0}),
\]
as claimed in Lemma~\ref{lem:C0_observability}.

\subsection{Spectral initialization (Lemmas~\ref{lem:si} and \ref{lem:wls})}\label{subsec:specinit}

Both results rely on the following lemma. 
\begin{lem}
\label{lem:dist_xz} Given $C_{0},C\in\C^{m\times n}$ and $x_{\gnd}\in\C_{1}^{n}$
such that $C_{0}x_{\gnd}=0$, let $x_{\wls}\in\arg\min_{x\in\C_{1}^{n}}\|Cx\|^{2},$
and $x_{\min}\in\arg\min_{x\in\C^{n},\|x\|^{2}=n}\|Cx\|^{2}.$ If
the difference $E\coloneqq C-C_{0}$ satisfies 
\[
\sigma_{n-1}(C_{0})>\opnorm E+\frac{1}{\sqrt{n}}\norm{Ex_{\gnd}},
\]
then the following holds for both $x=x_{\wls}$ and $x=x_{\min}$
\begin{equation}
\dist(x,x_{\gnd})\le2\sqrt{2}\left(\mu-\frac{\norm{Ex_{\gnd}}^{2}}{n\mu}\right)^{-1}\norm{Ex_{\gnd}},\label{eq:dist_xz}
\end{equation}
where $\mu\coloneqq\sigma_{n-1}(C_{0})-\opnorm E$. 
\end{lem}
\begin{IEEEproof}
Let $\mu\coloneqq\sigma_{n-1}(C_{0})-\|E\|_{\op}.$ Then, for every
$r\perp x_{\gnd}$, 
\[
\|Cr\|\ge\|C_{0}r\|-\|Er\|\ge\bigl(\sigma_{n-1}(C_{0})-\|E\|_{\op}\bigr)\|r\|=\mu\|r\|.
\]
Changing the global phase of $x$ if necessary, assume that $x_{\gnd}^{*}x=|x_{\gnd}^{*}x|$.
Using Gram--Schmidt orthogonalization, define $r\coloneqq x-\alpha x_{\gnd}$
where $\alpha\coloneqq x_{\gnd}^{*}x/\|x_{\gnd}\|^{2}$. Then,
\[
x=\alpha x_{\gnd}+r,\qquad x_{\gnd}^{*}r=0,\qquad0\le\alpha\le1.
\]
Since $\|x\|=\|x_{\gnd}\|=\sqrt{n}$, we have $\|r\|^{2}=n(1-\alpha^{2}),$
and
\[
\dist(x,x_{\gnd})=\|x-x_{\gnd}\|=\sqrt{2n(1-\alpha)}\le\sqrt{2}\|r\|.
\]
By optimality of $x$ and since $C_{0}x_{\gnd}=0$, 
\begin{align*}
0 & \le\|Cx_{\gnd}\|^{2}-\|Cx\|^{2}\\
 & =(1-\alpha^{2})\|Ex_{\gnd}\|^{2}-\|Cr\|^{2}-2\alpha\Re\inner{Ex_{\gnd}}{Cr}\\
 & \le\frac{\|r\|^{2}}{n}\|Ex_{\gnd}\|^{2}-\|Cr\|^{2}+2\|Ex_{\gnd}\|\|Cr\|\\
 & \le\left(\frac{\|Ex_{\gnd}\|^{2}}{n\mu^{2}}-1\right)\|Cr\|^{2}+2\|Ex_{\gnd}\|\|Cr\|
\end{align*}
where we used $\|Cr\|\ge\mu\|r\|$ in the final line. Because $\mu>\|Ex_{\gnd}\|/\sqrt{n}$
by hypothesis, rearranging gives 
\[
\|Cr\|\le2\left(1-\frac{\|Ex_{\gnd}\|^{2}}{n\mu^{2}}\right)^{-1}\|Ex_{\gnd}\|.
\]
Substituting $\|Cr\|\ge\mu\|r\|$ and $\dist(x,x_{\gnd})\le\sqrt{2}\|r\|$
proves the claim. 
\end{IEEEproof}
We will use this to prove Lemmas~\ref{lem:si} and also the following.
\begin{lem}
\label{lem:wls}Under the same conditions as Lemma~\ref{lem:si},
$x_{\wls}\coloneqq\arg\min_{x\in\C_{1}^{n}}\|Cx\|^{2}$ satisfies
\begin{align}
\dist(x_{\si},x_{\wls}) & \le12\sqrt{2}\,\frac{\|(C-C_{0})x_{\gnd}\|}{\sigma_{n-1}(C_{0})},\label{eq:cor_si_mle}\\
\dist(x_{\wls},x_{\gnd}) & \le4\sqrt{2}\,\frac{\|(C-C_{0})x_{\gnd}\|}{\sigma_{n-1}(C_{0})}.\label{eq:cor_mle}
\end{align}
\end{lem}
The proof of Theorem~\ref{thm:wls} follows by repeating the proof
of Theorem~\ref{thm:main_si} using Lemma~\ref{lem:wls}. 
\begin{IEEEproof}[Proof of Lemmas~\ref{lem:si} and \ref{lem:wls}]
Let $E\coloneqq C-C_{0}$, $\mu\coloneqq\sigma_{n-1}(C_{0})-\|E\|_{\op}$,
and $t\coloneqq\|Ex_{\gnd}\|/\mu$. Since $\|E\|_{\op}\le\frac{1}{3}\sigma_{n-1}(C_{0})$,
we have $\mu\ge\frac{2}{3}\sigma_{n-1}(C_{0})$. Moreover, since $\|x_{\gnd}\|=\sqrt{n}$,
\[
\frac{\|Ex_{\gnd}\|}{\sqrt{n}\,\mu}\le\frac{\|E\|_{\op}}{\mu}\le\frac{1}{2}.
\]
So $\mu>\|Ex_{\gnd}\|/\sqrt{n}$ holds, which is exactly the hypothesis
of Lemma~\ref{lem:dist_xz}. The same estimate also gives $t^{2}/n\le1/4$;
applying Lemma~\ref{lem:dist_xz} gives, for both $x=x_{\wls}$ and
$x=x_{\min}$, 
\begin{align*}
\dist(x,x_{\gnd})\le2\sqrt{2}\,\frac{t}{1-t^{2}/n} & \le\frac{8\sqrt{2}}{3}\,\frac{\|Ex_{\gnd}\|}{\mu}\\
 & \le4\sqrt{2}\,\frac{\|Ex_{\gnd}\|}{\sigma_{n-1}(C_{0})}.
\end{align*}
This proves (\ref{eq:cor_mle}) for $x_{\wls}$. It remains to pass
from $x_{\min}$ to $x_{\si}=\proj(x_{\min})$. Choose $\alpha\in\C_{1}$
such that 
\[
\dist(x_{\min},x_{\gnd})=\|\alpha x_{\min}-x_{\gnd}\|.
\]
Since entrywise projection commutes with global phase rotations, $\proj(\alpha x_{\min})=\alpha x_{\si}$.
Using the fact that projection onto $\C_{1}^{n}$ is $2$-Lipschitz
\cite[Lem.~20]{boumal2016nonconvex} yields
\begin{align*}
\dist(x_{\si},x_{\gnd}) & \le\|\alpha x_{\si}-x_{\gnd}\|\\
 & =\|\proj(\alpha x_{\min})-x_{\gnd}\|\le2\|\alpha x_{\min}-x_{\gnd}\|.
\end{align*}
Combining this with the bound for $x_{\min}$ proves (\ref{eq:main_mle}).
Finally, 
\[
\dist(x_{\si},x_{\wls})\le\dist(x_{\si},x_{\gnd})+\dist(x_{\wls},x_{\gnd}),
\]
which gives (\ref{eq:cor_si_mle}). 
\end{IEEEproof}

\subsection{Spectral certification (Lemma~\ref{lem:main_sdp})}

Recall that the certified bound $\delta(\cdot)$ and the slack matrix
$S(\cdot)$ were defined in (\ref{eq:delta_def}) and (\ref{eq:S_def}).
We first prove two lemmas. 
\begin{lem}
\label{lem:dual_gap}Given $C\in\C^{m\times n}$ and $x,z\in\C_{1}^{n}$,
we have
\[
x^{*}S(z)x=\|Cx\|^{2}-\|Cz\|^{2}.
\]
\end{lem}
\begin{IEEEproof}
For any $x\in\C_{1}^{n}$, 
\begin{align*}
x^{*}S(z)x & =x^{*}C^{*}Cx-x^{*}\Re[\ddiag(C^{*}Czz^{*})]x\\
 & =\norm{Cx}^{2}-\sum_{i=1}^{n}\Re\!\left((C^{*}Cz)_{i}\conj(z_{i})\right),\\
 & =\norm{Cx}^{2}-\Re(z^{*}C^{*}Cz).
\end{align*}
\end{IEEEproof}
\begin{lem}
\label{lem:strict_comp}Given $C_{0},C\in\C^{m\times n}$ and $x_{\gnd}\in\C_{1}^{n}$
such that $C_{0}x_{\gnd}=0$, let 
\[
x_{\wls}\in\arg\min_{x\in\C_{1}}\|Cx\|^{2}.
\]
If the difference $E\coloneqq C-C_{0}$ satisfies 
\[
\opnorm E\le\half\rho,\qquad\norm{Ex_{\gnd}}\le\frac{1}{4(\kappa_{\infty}+1)}\rho,
\]
where 
\[
\kappa_{\infty}:=\frac{\max_{i}\norm{C_{0}e_{i}}}{\sigma_{n-1}(C_{0})},\qquad\rho\coloneqq\sigma_{n-1}(C_{0}),
\]
then the dual matrix satisfies 
\[
\lambda_{n-1}(S(x_{\wls}))\ge\frac{\rho^{2}}{8(\kappa_{\infty}+1)}.
\]
\end{lem}
\begin{IEEEproof}
Write $x=x_{\wls}$ and $S=S(x_{\wls})$. Weyl's inequality gives
\[
\lambda_{n-1}(S)\ge\lambda_{n-1}(C^{*}C)-\norm{\Re[\ddiag(C^{*}Cxx^{*})]}_{\op}.
\]
The diagonal term satisfies 
\begin{align}
\norm{\Re[\ddiag(C^{*}Cxx^{*})]}_{\op} & \le\norm{\ddiag(C^{*}Cxx^{*})}_{\op}\nonumber \\
 & =\norm{C^{*}Cx}_{\infty}\nonumber \\
 & \le\norm{C^{*}}_{2\to\infty}\norm{Cx}.\label{eq:diag_bound_wls}
\end{align}
Where $\norm{\cdot}_{2\to\infty}$ denotes the induced norm from $\ell_{2}$
to $\ell_{\infty}$. By optimality of $x$ and $C_{0}x_{\gnd}=0$,
\[
\norm{Cx}\le\norm{Cx_{\gnd}}=\norm{Ex_{\gnd}}.
\]
Plugging this into (\ref{eq:diag_bound_wls}) yields 
\begin{align}
\lambda_{n-1}(S) & \ge\lambda_{n-1}(C^{*}C)-\norm{C^{*}}_{2\to\infty}\norm{Ex_{\gnd}}.\label{eq:SDP_part2_1}
\end{align}
Substituting $\opnorm E\le\half\rho$ yields, 
\[
\lambda_{n-1}(C^{*}C)\ge\left(\sigma_{n-1}(C_{0})-\opnorm E\right)^{2}\ge\frac{\rho^{2}}{4}.
\]
Moreover, 
\begin{align*}
\norm{C^{*}}_{2\to\infty} & =\max_{i}\norm{(C_{0}+E)e_{i}}\\
 & \le\kappa_{\infty}\rho+\opnorm E\le\left(\kappa_{\infty}+\half\right)\rho.
\end{align*}
Substituting the above and $\norm{Ex_{\gnd}}\le\frac{\rho}{4(\kappa_{\infty}+1)}$
into (\ref{eq:SDP_part2_1}) gives 
\begin{align*}
\lambda_{n-1}(S) & \ge\frac{\rho^{2}}{4}-\left(\kappa_{\infty}+\half\right)\rho\norm{Ex_{\gnd}}\\
 & \ge\frac{\rho^{2}}{4}\left[1-\frac{\kappa_{\infty}+\half}{\kappa_{\infty}+1}\right]=\frac{\rho^{2}}{8(\kappa_{\infty}+1)}.
\end{align*}
\end{IEEEproof}
\begin{IEEEproof}[Proof of Lemma~\ref{lem:main_sdp}]
Write $S_{\wls}\coloneqq S(x_{\wls})$. By Lemma~\ref{lem:dual_gap}
and the optimality of $x_{\wls}$, 
\begin{equation}
x^{*}S_{\wls}x=\|Cx\|^{2}-\|Cx_{\wls}\|^{2}\ge0\qquad\text{for all }x\in\C_{1}^{n}.\label{eq:dual_gap}
\end{equation}
Equality holds at $x=x_{\wls}$. Together with Lemma~\ref{lem:strict_comp},
which gives $\lambda_{n-1}(S_{\wls})>0$, this implies 
\begin{equation}
S_{\wls}\succeq0,\qquad\ker(S_{\wls})=\mathrm{span}\{x_{\wls}\}.\label{eq:strict_comp}
\end{equation}
Now suppose $\delta(x)=0$. Proposition~\ref{prop:cert} gives $\|Cx\|^{2}=\|Cx_{\wls}\|^{2}$,
so (\ref{eq:dual_gap}) yields $x^{*}S_{\wls}x=0$. By (\ref{eq:strict_comp}),
this is equivalent to $x=x_{\wls}\exp(\i\theta_{0})$, and hence $\dist(x,x_{\wls})=0$.
Conversely, if $\dist(x,x_{\wls})=0$, then $x=x_{\wls}\exp(\i\theta_{0})$
and $xx^{*}=x_{\wls}x_{\wls}^{*}$. Hence $S(x)=S_{\wls}$, and (\ref{eq:strict_comp})
gives $\delta(x)=-n\lambda_{\min}(S(x))=-n\lambda_{\min}(S_{\wls})=0.$
\end{IEEEproof}

\section*{Acknowledgments}

IG and RZ were supported by NSF CAREER Award ECCS-2047462 and ONR
Award N00014-24-1-2671. AM was supported by the Hi! PARIS and ANR/France
2030 program (ANR-23-IACL-0005) and the Swiss SERI (contract MB22.00027).

\bibliographystyle{ieeetr}
\bibliography{references}

\end{document}